\documentclass[11pt, one side,emlines]{amsart}
\usepackage{amssymb,latexsym,xy,eucal,mathrsfs, graphicx, tikz, amsmath, caption}
\textwidth=15.2cm \textheight=24cm \theoremstyle{plain}
\newtheorem{lemma}{Lemma}[section]
\newtheorem{theorem}[lemma]{Theorem}

\newtheorem{corollary}[lemma]{Corollary}

\theoremstyle{definition}
\newtheorem{example}[lemma]{Example}
\newtheorem{remark}[lemma]{Remark}

\newtheorem{definition}[lemma]{Definition}

\numberwithin{equation}{section}  \voffset
-55truept \hoffset -15truept
\begin{document}
\baselineskip 15truept
\title{Conrad's Partial Order on p.q.-Baer $*$-Rings}
 \maketitle 
 \begin{center} 
Anil Khairnar\\
 {\small Department of Mathematics, Abasaheb Garware College, Pune-411004, India.\\
  \email{\emph{anil.khairnar@mesagc.org; anil\_maths2004@yahoo.com}}}\\  
  \vspace{.4cm}
  B. N. Waphare\\  
 {\small Center for Advanced Studies in Mathematics,
Department of Mathematics,\\ Savitribai Phule Pune University, Pune-411007, India.}\\
 \email{\emph{bnwaph@math.unipune.ac.in; waphare@yahoo.com}} 
 \end{center} 
\begin{abstract}    
  We prove that a p.q.-Baer $*$-ring forms a pseudo lattice with Conrad's partial order and also characterize
  p.q.-Baer $*$-rings which are lattices. 
 The initial segments of a p.q.-Baer $*$-ring with Conrad's partial order 
are shown to be orthomodular posets.
\end{abstract}
 \section{Introduction}        
                A {\it $*$-ring} $R$ is a ring equipped with an involution $x \rightarrow x^* $,
                that is an additive anti-automorphism of period at most two. 
                An element $e$ of
                a $*$-ring $R$ is a {\it projection} if $e=e^2$ (idempotent) and $e=e^*$ (self adjoint). For
                a nonempty subset $B$ of $R$, we write $r_R(B)=\{x\in R ~|~bx=0,~for~ every~ b\in B\}$, and call the {\it right annihilator} of $B$
            in $R$. Similarly, we define the {\it left annihilator} of $B$ in $R$ (denoted by $l_R(B)$).   
            A ring is said to be abelian if its every idempotent is central. 
            A ring without nonzero nilpotent elements is called a reduced ring. 
            Let $P$ be a poset and $a,b \in P$, then the join of $a$ and $b$, denoted by
                           $a\vee b$ is defined as $a \vee b = \sup~ \{a, b\}$ and the meet of $a$ and $b$,
                           denoted by $a\wedge b$ is defined as $a \wedge b= inf~ \{a, b\}$.
                           A poset $P$ is said to be a pseudo lattice, if for $a, b \in P$, whenever $a, b$ have a common upper bound, then $a\wedge b$ and $a \vee b$ both exist. \\
              \indent  Kaplansky \cite{Kap} introduced Baer rings and Baer $*$-rings to
            abstract various properties of $AW^*$ algebras, von Neumann
            algebras and complete $*$-regular rings.
                 The concept of a Baer $*$-ring is naturally motivated in the study of functional analysis. 
               Early motivation for studying rings with involution came from rings of operators.\\               
                       \indent  The set of projections in a Rickart $*$-ring $R$ forms an orthomodular lattice under the partial order `$e\leq_p f$ if and only if $e=fe=ef$'. This lattice is extensively studied in \cite{Ber, Kap, ms2}.       
                               In \cite{Bak, Doli, Doli2, Hart, Sem} 
                               the authors studied partial orders on complex matrices or $ \mathscr{B}(H)$
                               (the algebra of all bounded linear operators on an infinite-dimensional Hilbert space $H$). 
                               In \cite{d, j, Kre} the authors studied partial orders on Rickart $*$-rings.                        
  Hartwig \cite{Hart} defined the plus partial order on the set of regular elements in a semigroup.                                   
 For $m \times n$ matrices over a division ring $D$ (that is $D_{m\times n}$). Hartwig \cite{Hart} 
 use the concept of rank $\rho(.)$ and obtained the following result, which characterized the plus order for the 
 ring $D_{m\times n}$.
 \begin{theorem} [Theorem 2, \cite{Hart}] \label{th001} Let $A, B \in D_{m \times n}$. Then $A \leq B$ 
 if and only if $\rho(B-A) = \rho(B)-\rho(A)$. In particular, rank-subtractivity is a partial-ordering relation 
 on $D_{m \times n}$.
 \end{theorem} 
 Also in the same paper \cite{Hart}, Hartwig posed the following open problems.\\
 {\bf Problem 1:} Can one induce a partial ordering on a ring $R$, by a subtractive rank-like function 
 $\rho:R\rightarrow G$, where $G$ is a well-ordered abelian group and $\rho(b-a)=\rho(b)-\rho(a)$?   \\
{\bf Problem 2:} Does $a\leq c$, $b\leq c$, $aR\cap bR=\{0\}=Ra\cap Rb\Rightarrow a+b \leq c$?\\
 \indent Conrad \cite{Con} extended the work of Abian \cite{Ab} by showing that a ring $R$ is 
          partially ordered by the relation $a\leq_c b$ if and only if $arb=ara$ for all $r\in R$ 
          (this is called Conrad's relation) precisely when 
          it is semiprime.         
          Burgess and Raphael \cite{Bur} proved that this relation, when defined on a semigroup $S$, is
          a partial order whenever $S$ is weakly separative. \\         
        \indent  Birkenmeier et al. \cite{Bir4} introduced principally quasi-Baer
         (p.q.-Baer) $*$-rings.  A $*$-ring $R$ is said to be {\it a p.q.-Baer $*$-ring} if,
   for every principal right ideal $aR$ of $R$, $r_R(aR)=eR$, where $e$ is a projection in
$R$. From the above definition, it follows that $l_R(aR)=Rf$ for a suitable projection $f$.
There is an abelian p.q.-Baer $*$-ring which is not a Rickart $*$-ring. Also,
 reduced Rickart $*$-rings are p.q.-Baer $*$-rings. In \cite{Bir4}, Birkenmeier et al. have given examples  p.q.-Baer $*$-rings those are not Rickart $*$-rings or quasi-Baer $*$-rings. It is easy to observe that p.q.-Baer $*$-rings are semiprime. Therefore Conrad's relation is a 
 partial order on a p.q.-Baer $*$-ring.\\
\indent  Let $R$ be a $*$-ring and $x \in R$, we say that $x$
            possesses a {\it central cover} if there exists a smallest central
            projection $h$ such that $hx=x$. If such a projection $h$ exists,
            then it is unique, and is called the central cover of $x$, denoted
            by $h=C(x)$ (see \cite{Ber}). In \cite{Anil} the authors proved the existence of central cover of every element of a p.q.-Baer $*$-ring. 
 In the second section of this paper, we characterize Conrad's partial order on p.q.-Baer $*$-rings in terms of central covers. Also, we prove the result similar to Theorem \ref{th001}, in connection to Problem 1.\\
 \indent In \cite{Black}, Blackwood et al. answered Problem 2 negatively for the minus partial order on the ring of matrices over a field. In the third section, we answer Problem 2 positively, for p.q.-Baer $*$-rings with Conrad's partial order.\\
  \indent Janowitz \cite{j} proved that the initial segments of an arbitrary Rickart $*$-ring
  with the $*$-order are orthomodular posets. The same result is proved by Kr$\bar{e}$mere \cite{Kre}
  for the left-star order. In the last section, we prove that the initial segments of a p.q.-Baer $*$-ring
  with Conrad's partial order 
  are orthomodular posets. 
  \section{Conrad's Relation on p.q.-Baer $*$-rings}  
  In this section, we characterize Conrad's partial order on a p.q.-Baer $*$-ring in terms of  
  central covers of elements. Also, we construct a subtractive function in terms of central covers, 
  which induces Conrad's partial order on a p.q.-Baer $*$-ring.
  \begin{remark} \label{rm101}Let $R$ be a $*$-ring and $P(Z(R))$ denotes the set of central projections of $R$.
  \begin{enumerate}
  \item If we restrict the Conrad's relation to the set $P(Z(R))$,
  then the relation becomes a partial order on $P(Z(R))$.
  Further, for $e,f \in P(Z(R))$, $e \leq f$ if and only if $e=ef$.
  \item For any $e\in P(Z(R))$ the central cover $C(e)$ exists and $C(e)=e$. 
  Moreover, whenever $C(x)$ exists for some $x\in R$, then for any $e \in P(Z(R))$, 
  the central cover $C(ex)$ exists and $C(ex)=e C(x)$.
  \item Let $a \in R$. If $C(a)$ exists in $R$, then $C(a^*)$ exists in $R$ and $C(a^*)=C(a)$ (see \cite{Anil}).
    \end{enumerate}  
  \end{remark} 
  Hence fourth, $\leq$ denotes the Conrad's partial order on p.q.-Baer $*$-ring.
\begin{lemma} \label{lm201}  Let $R$ be a $*$-ring and $x \in R$. Let $e\in R$ be a central projection in $R$ such that
 (1) $xe=x$ and (2) $xRy=0$ implies $ey=0$. Then $e=C(x)$. 
\end{lemma}
 \begin{proof} To prove that $e=C(x)$, it is sufficient to prove that  $e$ is the smallest central projection with $xe=x$.
    Let $e'\in R$ be a central projection such that $xe'=x$. Then $x(1-e')=0$. Since $1-e'$ is central,
     $xR(1-e')=0$. By condition (2), we have $e (1-e')=0$ and hence $e=ee'$.
    Therefore $e \leq e'$. Thus $e=C(x)$.  
 \end{proof}
  \indent  The following result \cite{Anil}, 
   gives the existence of central cover of every element in a p.q.-Baer $*$-ring.
    \begin{theorem} [Theorem 2.3, \cite{Anil}] \label{p1th101} Let $R$ be a p.q.-Baer $*$-ring and $x\in R$.
              Then $x$ has a central cover $e \in R$. Further,
              $xRy=0$ if and only if $yRx=0$ if and only if $ey=0$.\\
              That is $r_R(xR)=r_R(eR)=l_R(Rx)=l_R(Re)=(1-e)R=R(1-e)$.
  \end{theorem}  
         In the following theorem we characterize the  Conrad's relation in terms of central cover.         
               \begin{lemma} \label{th201} Let $R$ be a p.q.-Baer $*$-ring and $a, b \in R$. Then the following statements are equivalent.
               \begin{enumerate}
                \item $a^*rb=a^*ra$ for all $r \in R$.
                \item $a=C(a) b$.
                 \item $arb=ara$ for all $r \in R$ (that is $a \leq b$).               
               \end{enumerate}                
               \end{lemma}
                \begin{proof} $(1)\Rightarrow (2)$: By (1), $a^*r(b-a)=0$ for all $r \in R$, that is $a^*R(b-a)=0$. 
                By Theorem \ref{p1th101}, we get $C(a^*) (b-a)=0$.
                This implies $C(a) (b-a)=0$ (by Remark \ref{rm101}). Consequently, $a= C(a) b$.\\
                $(2)\Rightarrow (3)$: For $r \in R$, we have by (2), $ara=arC(a)b=C(a)arb=arb$. 
                Therefore $arb=ara$ for all $r \in R$.\\
                 $(3)\Rightarrow (1)$:  By the similar arguments as in the proof of $(1)\Rightarrow (2)$, we get $a= C(a) b$.
                 Further, for $r\in R$, $a^*ra=a^*rC(a)b=C(a)a^*rb=C(a^*)a^*rb=a^*rb$.  
                  Thus $a^*rb=a^*ra$ for all $r \in R$.
                   \end{proof} 
           The above lemma essentially says that, in a p.q.-Baer $*$-ring $R$, for $a, b \in R$, $a\leq b$ if and only if $a=C(a) b$. Therefore, we use the relation $a=C(a) b$ as Conrad's relation (partial order) on a p.q.-Baer $*$-ring.          
        The following lemma leads to the result which constructs a subtractive function on a p.q.-Baer $*$-ring.                
        \begin{lemma} \label{lm202} Let $R$ be a p.q.-Baer $*$-ring and $a, b \in R$ be such that $a \leq b$. Then,
               \begin{enumerate}
                \item $C(a)\leq C(b)$ and $a=a C(b)= b C(a)$
                \item $C(b-a)=C(b) - C(a)$.
                 \end{enumerate} 
        \end{lemma}
        \begin{proof} $(1)$: As $a \leq b$, $a=C(a) b$ and hence $C(a)=C(C(a) b) =C(a) C(b)$ (by Remark \ref{rm101}). This yields $C(a)\leq C(b)$.
        By multiplying by $a$ to both sides of $C(a)=C(a)C(b)$ we get, $a =a C(b)$. Therefore $a=a C(b)= b C(a)$.\\
                                 \noindent $(2)$: As, $C(a) \leq C(b)$, $C(b)-C(a)$ is a central projection with $(b-a)(C(b)-C(a))=b C(b) - b C(a) - a C(b) + a C(a)=b - a-a+a =b-a$ (by part (1)).
                Further for $y \in R$, $(b-a)Ry=0$ if and only if 
                $bry=ary$ for all $r \in R$ if and only if
                $bC(b)ry=bC(a)ry$ for all $r \in R$ if and only if
                $bR(C(b)-C(a))y=0$ if and only if 
                $C(b) (C(b)-C(a))y=0 $ (by Theorem \ref{p1th101})
                if and only if $(C(b)-C(a))y=0$.
                Thus, by Lemma \ref{lm201}, $C(b-a)= C(b)-C(a)$, as required.  
        \end{proof}
      In the above lemma we have proved that in a p.q.-Bear $*$-ring $R$, for $a, b \in R$, if $a\leq b$ then $C(b-a)=C(b)-C(a)$. The following lemma gives a sufficient condition so that the converse of this statement is true.  
        \begin{lemma} \label{nth202} Let $R$ be a p.q.-Baer $*$-ring in which $2$ is invertible. Let $a, b \in R$  
       be such that $C(b-a)=C(b)-C(a)$. Then  $a \leq b$.         
        \end{lemma}
        \begin{proof} Let $a, b \in R$ be such that $C(b-a)=C(b)-C(a)$. Then $(C(b)-C(a))^2=C(b)-C(a)$, which yields $2 C(b) C(a) = 2 C(a)$. Since $2$ is invertible element in $R$, we have $C(b) C(a) = C(a)$. 
       Further, $C(b-a)C(a)=(C(b)-C(a))C(a)=0$. By Theorem \ref{p1th101}, $(b-a)RC(a)=0$.
        Consequently, $(b-a)C(a)=0$ and hence $bC(a)=a$. Therefore $a \leq b$.         
        \end{proof}
        The following theorem characterises Conrad's partial order in terms of central covers, which gives a result similar to Theorem \ref{th001}.
         \begin{theorem} \label{ncr202} Let $R$ be a p.q.-Baer $*$-ring in which $2$ is invertible and let $a, b \in R$. Then  $a \leq b$ if and only if $C(b-a)=C(b)-C(a)$.          
         \end{theorem}
         \begin{proof} The proof follows from Lemmas \ref{lm202} and \ref{nth202}.         	
       	\end{proof}     
  \section{When does a p.q.-Baer $*$-Ring become a Lattice?}     
    Hartwig \cite{Hart1} showed that a star-regular ring $R$ forms a pseudo upper semilattice under the star-orthogonal partial ordering. That is, for every $a, b \in R$, $a, b$ have a common upper bound if and only if $a \vee b$ exists in $R$.
    In this section, we prove that, in a p.q.-Baer $*$-ring $R$ with Conrad's partial order, for every $a, b \in R$, $a, b$ have a common upper bound if and only if $a \vee b$ exists in $R$. Also, we give a sufficient condition for p.q.-Bear $*$-ring to be a lattice. As a consequence, we answer  Problem 2 positively for p.q.-Baer $*$-rings with Conrad's partial order. 
    \begin{definition} 
    In \cite{Con}, a concept of orthogonality is introduced as follows. Let $R$ be a semiprime ring and $a, b \in R$. Then $a$ is said be {\it orthogonal} to $b$ if $aRb=0$. In a p.q.-Baer $*$-ring this condition is equivalent to $C(a)C(b)=0$ (see \cite{Anil}). We write $a \perp b$, whenever a orthogonal to $b$. 
    \end{definition}            
    \indent Recall the following definition and theorem from \cite{Bur}. 
   \begin{definition}  Let $R$ be a semiprime ring. 
     For an ideal $I$ of $R$, $Ann~I~=\{r\in R ~|~ rI=0  \}$. 
     If for each ideal $I$, $Ann~I$ contains a nonzero central idempotent then $R$ is called {\it weakly $i$-dense}.
     $R$ is {\it orthogonally complete} if every orthogonal set has a supremum. 
    \end{definition}      
     \begin{theorem} [Theorem 9, \cite{Bur}] An orthogonally complete semiprime ring which is weakly $i$-dense 
     is complete.     
     \end{theorem}    
     We give an example of a commutative, reduced, weakly $i$-dense p.q.-Baer $*$-ring 
     which is not orthogonally complete. 
     \begin{example} Let $R=\{x \in \prod_{i=1}^{\infty} Q ~|~$ for almost all $i$, $x_i \in \mathbb Z  \}$.
           Then $R$ is a commutative $*$-ring with an identity involution. For $a=(a_1,a_2,\cdots)\in R$, 
           $r_R(a)=bR$ where $b=(b_1,b_2, \cdots)$ with $b_i=1$ if $a_i=0$; and $b_i=0$ if $a_i \neq 0$.
           Note that $b^2=b=b^*$. Therefore $R$ is a Rickart $*$-ring. 
           Since an abelian Rickart $*$-ring is a reduced p.q.-Baer $*$-ring, $R$ becomes a commutative reduced
           p.q.-Baer $*$-ring. Since every ideal of $R$ is a principal ideal and $R$ is a p.q.-Baer $*$-ring, therefore  by Theorem \ref{p1th101}, $R$ is weakly $i$-dense.
           Let $c_1=(\frac{1}{2},0,0,\cdots), c_2=(0,\frac{1}{2},0,0,\cdots), \cdots $, and 
           $S=\{c_n ~|~n \in \mathbb N \}$. Then $S$ is an orthogonal subset of $R$ which does not have
           supremum in $R$. Thus $R$ is not orthogonally complete.       
     \end{example} 
     In the following theorem, we prove that a p.q.-Baer $*$-ring forms a pseudo lattice with respect to Conrad's partial order.
  \begin{theorem} \label{th205} Let $R$ be a p.q.-Baer $*$-ring and $a, b \in R$ have a common upper bound. Then 
  \begin{enumerate}
  \item $aC(b)=b C(a)$;
   \item $a^*rb=C(a)b^*rb=C(b) a^*ra$ for all $r \in R$. Hence, $a^*b$ is self adjoint;
   \item $arb^*=C(a)brb^*=C(b)ara^*$ for all $r\in R$. Hence, $ab^*$ is self adjoint;
   \item $a \wedge b =a C(b)=b C(a)$; and
   \item $a \vee b=a+b- a \wedge b$. 
  \end{enumerate}  
  \end{theorem}
  \begin{proof} Let $a, b, c \in R$ and $c$ be a common upper bound of $a$ and $b$. Then $a=C(a) c$ and $b=C(b)c $. 
            By Lemma \ref{th201}, $a^*ra=a^*rc$, $b^*rb=b^*rc$ for all $r \in R$. Also, $b^*rb=c^*rb$ for all $r \in R$.\\
            $(1)$: Since $a=C(a) c$ and $b=C(b)c $, we have $aC(b)=C(a) c C(b)=b C(a)$. \\
          $(2)$: Let $r\in R$. Then $a^*rb=a^*rC(b)c=C(b)a^*rc=C(b)a^*ra$.   
            Also, $a^*rb=(C(a)c)^*rb=C(a)c^*rb=C(a)b^*rb$. 
            Consequently, $a^*rb=C(a)b^*rb=C(b) a^*ra$ for all $r \in R$. 
            In particular for $r=1$, we have $a^*b=C(b)a^*a$. Therefore $(a^*b)^*=C(b)a^*a=a^*b$. Thus $a^*b$ is self adjoint.  \\
          $(3)$: The proof is similar to the proof of part $(1)$.\\
          $(4)$: To prove $a\wedge b = a C(b)$, first we prove that $a C(b)$ is a common lower bound of $a$ and $b$. By Remark \ref{rm101}, $C(aC(b)) a=C(a)C(b)a=a C(b)$. This implies $a C(b) \leq a$. Similarly, $bC(a) \leq b$.
          By part $(1)$, we get $aC(b) \leq b$.
          Let $d \in R$ be such that $d \leq a$ and $d \leq b$. Then $d=C(d)a=C(d)b$ and hence $d C(b)=C(d) b$. Further, $C(d) a C(b)=d C(b)=C(d)b=d$. Therefore $d \leq a C(b)$. Thus $a \wedge b =a C(b)=b C(a)$.  \\
          $(5)$: By $(1)$ and $(4)$, $C(a) (a+b - a\wedge b) = C(a) (a + b - a C(b))= a C(a) + b C(a) - a C(a) C(b)=
          a+bC(a)-a C(b) =a$. This yields $a \leq (a+b - a \wedge b)$. Similarly, $b \leq (a+b - a \wedge b)$.
          Let $d \in R$ be such that $a \leq d$ and $b \leq d$. Then $a=C(a) d$ and $ b = C(b) d$.
         Let $r\in R$.  By using part $(2)$, $(a+b-a \wedge b)^*r (a+b-a\wedge b)=(a^*+b^*-a^*C(b)) r (a+b-aC(b))
         =a^*ra+a^*rb-a^*raC(b)+b^*ra+b^*rb-b^*raC(b)-a^*raC(b)-a^*rbC(b)+a^*raC(b)
         =a^*ra + a^*rb - a^*rb + b^*ra + b^*rb - C(b)b^*ra - a^*rb - a^*raC(b) + a^*raC(b)
         =a^*ra + b^*ra + b^*rb - b^*ra - a^*rb =a^*rdC(a) + b^*rdC(b) - a^*rdC(b)
         =a^*rd + b^*rd - a^*rd C(b) = (a^*+b^*-a^*C(b))rd=(a+b-aC(b))^*rd=(a+b- a\wedge b)^*rd$.
         Thus, by Lemma \ref{th201}, $(a+b- a\wedge b) \leq d$.
         Therefore $a\vee b = a+b - a\wedge b$.         
    \end{proof}
    As an immediate consequence of above theorem we have following corollaries.
    \begin{corollary} Let $R$ be a p.q.-Baer $*$-ring. Then $R$ is a pseudo lattice with respect to Conrad's partial order.    
    \end{corollary}
        \begin{corollary} Let $R$ be a p.q.-Baer $*$-ring and $a, b \in R$. If $a\vee b$ exists in $R$ then 
        $a \vee b =a+b(1-C(a)) = b+a(1 -C(b))$.        
        \end{corollary}
       By Theorem \ref{th205}(1), in a p.q.-Baer $*$-ring $R$, if $a, b \in R$ have a common upper bound then $aC(b)=bC(a)$. 
        In the following lemma, we prove that the converse of this statement is also true.
         \begin{lemma} \label{nlm401} Let $R$ be a p.q.-Baer $*$-ring and $a, b \in R$. 
          If $a C(b)= b C(a)$ then $a, b$ have a common upper bound.          
         \end{lemma}
        \begin{proof} Let $a, b \in R$ be such that $aC(b)=bC(a)$. We prove that $a+b - a C(b)$ is a common upper bound of $a$ and $b$.        
               We have $C(a)(a+b-a C(b))=a+C(a)b - a C(b)=a$. Also,
               $C(b)(a+b-aC(b))=a C(b) +  b - a C(b)=b$. Therefore $a \leq a+b - aC(b)$ and $b \leq a+b - aC(b)$, as required.
        \end{proof}
    \begin{corollary} Let $R$ be a p.q.-Baer $*$-ring and $a, b \in R$. Then, $a C(b)= b C(a)$ if and only if $a \vee b =a+b- a \wedge b$    
    \end{corollary}    
        
        The following theorem, characterises p.q.-Baer $*$-rings which form lattices with Conrad's partial order.
        \begin{theorem} Let $R$ be a p.q.-Baer $*$-ring. 
        Then $R$ is a lattice with respect to Conrad's partial order if and only if $aC(b)=bC(a)$ for all $a, b \in R$.        
        \end{theorem}
        \begin{proof} The proof follows from Theorem \ref{th205} and Lemma \ref{nlm401}.       
        \end{proof}
        We conclude this section with positive answer to  Problem 2, when $R$ is a p.q.-Baer $*$-ring 
       with Conrad's partial order. 
       \begin{theorem} Let $R$ be a p.q.-Baer $*$-ring and $a, b, c \in R$. 
       If $a\leq c,~b \leq c, ~ aR \cap bR =\{0\}$ then $a+b \leq c$.       
       \end{theorem}
       \begin{proof} Let $a,b, c \in R, a\leq c,~b \leq c$ and $aR\cap bR=\{0\}$. Then, by Theorem \ref{th205}, $aC(b)=b C(a)$.
       This implies $aC(b)\in aR\cap bR$ and hence $aC(b)=0$.    Again, by using Theorem \ref{th205}, $a\vee b = a+b$. Thus $a+b \leq c$.       
       \end{proof}
            
         \section{Orthogonality Relation on p.q.-Baer $*$-Rings}     
       In this section, we prove that the initial segments of an arbitrary p.q.-Baer $*$-rings with Conrad's partial order are orthomodular posets. \\
     \indent  We recall the following definitions from \cite{Cir}.  \\      
      A binary relation $\perp$ on a poset $(P, \leq, 0)$, where $0$ is the least element of the poset,
    is called an {\it orthogonality relation} (for the order $\leq$) if for all $x, y, z \in P$,
    \begin{enumerate}
     \item  if $x \perp y$, then $y \perp x$; 
     \item if $x \leq y$ and $y \perp z$, then $x \perp z$; and
    \item $0 \perp x$. \\    
          A poset with orthogonality $(P,~\leq,~\perp,~0)$ is called {\it quasi-orthomodular} if for all $x, y\in P$,
    \item if $x \perp y$, then $x \vee y$ exists;
    \item if $x \leq y$, then $y=x \vee z$ for some $z\in P$ with $x \perp z$;
   \item if $x \perp y$, $x \perp z$ and $y \leq x \vee z$, then $y \leq z$.     
     \end{enumerate} 
 A poset $(P, ~\leq,~ 0,~ 1)$ (where $0$ is the least and $1$ is the greatest element) is called an
  {\it orthocomplemented poset} if there is an operation $^{\perp}: P \rightarrow P$ such that for all $a, b \in P$, 
  \begin{enumerate}
   \item $a \wedge a^{\perp}$ and $ a \vee a^{\perp}$ exist, and $a\wedge a^{\perp}=0$ and $a \vee a^{\perp}=1$;
    \item$(a^{\perp})^{\perp}=a$;
    \item if $a \leq b$, then $b^{\perp}\leq a^{\perp}$.
 \end{enumerate}
      The operation $^{\perp}$ is called an orthocomplementation. In an orthocomplemented poset, 
      we define the relation $\perp$ by $a \perp b$ if and only if $a \leq b^{\perp}$. This is an orthogonality relation.      
      An orthocomplemented poset $(P,~\leq,~^\perp,~0,~1)$ is called orthomodular if for all $a, b\in P$,
      \begin{enumerate}
       \item if $a \perp b$, then $a \vee b$ exist;
       \item if $a\leq b$, then there exists an element $c\in P$ such that $c \leq a^{\perp}$ and $b=a \vee c$.
      \end{enumerate}
    Between orthomodularity and quasi-orthomodularity, the following connection holds.
    \begin{theorem}[\cite{Cir}] \label{th202}  In a quasi-orthomodular poset $(P,~\leq,~^\perp)$, all initial segments 
    $[0,p]=\{a\in  P~|~a \leq p  \}$ are orthomodular for some orthogonality $\perp_p$ on $([0,p], \leq)$.
    Furthermore, if $\perp_p$ is the orthogonality of the initial segment $[0,p]$, then for all $a, b \in [0,p]$,
    $a\perp_p b$ if and only if $a \perp b$. Moreover, if $x \perp_p y$ and $x, y \leq q$, then $x\perp_q y$.     
    \end{theorem}
    By using above theorem, we prove that the initial segments of  p.q.-Baer $*$-rings with Conrad's partial order are orthomodular posets, for that we prove the following sequence of theorems and lemmas.    
    \begin{theorem} The relation $\perp$ is an orthogonality relation on a p.q.-Baer $*$-ring.     
    \end{theorem}
    \begin{proof} Let $R$ be a p.q.-Baer $*$-ring. By definition of orthogonal elements, it is clear that for any $x, y \in R$, if $x \perp y$ then $y \perp x$.
    Suppose $a \leq b$ and $b \perp c$. Then $a=C(a) b$ and $C(b)C(c)=0$.
    By Lemma \ref{lm202},  $C(a)C(c)=C(a)C(b)C(c)=0$ and hence $a \perp c$.
    Further, $C(0)=0$, therefore $C(0)C(x)=0$ for any $x \in R$. Consequently, $0 \perp x$ for any $x \in R$.
    Thus the relation $\perp$ is an orthogonality relation. 
    \end{proof}
     \begin{lemma} \label{lm203'} Let $R$ be a p.q.-Baer $*$-ring. If $a$ and $b$ are orthogonal elements of $R$, then $a$ and $b$ have a common upper bound.     
        \end{lemma}
    \begin{proof} Let $a,b \in R$ be such that $a \perp b$. Then $C(a) C(b)=0$. This implies $aC(b)=C(a)b=0$. Therefore by Lemma \ref{nlm401}, $a$ and $b$ have a common upper bound.
      \end{proof}   
     In the following theorem, we prove that orthogonal elements of a p.q.-Baer $*$-ring possess the join and the meet.
   \begin{theorem}  \label{th203} Let $R$ be a p.q.-Baer $*$-ring and $a, b \in R$ be orthogonal elements. 
   Then $ a\wedge b$, $ a \vee b $ exist and $a\wedge b =0$, $a \vee b= a+b$.    
   \end{theorem}
   \begin{proof} Let $a, b \in R$ be orthogonal elements. Then by Lemma \ref{lm203'}, $a$ and $b$ have a common upper bound.
   By Theorem \ref{th205}, $a\wedge b =a C(b)$ and $a \vee b= a+b- aC(b)$. Since $a$ and $b$ are orthogonal elements, we have $aC(b)=0$. Therefore $a\wedge b =0$ and $a \vee b= a+b$.     
   \end{proof}
   The following lemma leads to the orthomodularity condition in a poset.
     \begin{lemma} \label{lm204}  Let $R$ be a p.q.-Baer $*$-ring and $a,b \in R$. If $a \leq b$ then there exists $c\in R$ such that
       $a \perp c$ and $b=a+c$.       
     \end{lemma}
     \begin{proof}  Let $a, b \in R$ and $a \leq b$. Then $a=C(a)b$ and hence $C(a)=C(a)C(b)$. Let $c=b-a$. By Lemma \ref{lm202},  $C(a)C(c)=C(a)C(b-a)=C(a)(C(b)-C(a))=C(a)C(b)-C(a)=0$. 
    Therefore  $a \perp c$.      
     \end{proof}
    \begin{lemma} \label{lm205} Let $R$ be a p.q.-Baer $*$-ring and $a, b, c \in R$. If $a \perp b$, $a \perp c$ and $ b \leq a \vee c$, 
    then $b \leq c$.     
    \end{lemma}
    \begin{proof} Let $a, b, c \in R$ be such that $a \perp b$, $a \perp c$ and $ b \leq a \vee c$. Then $C(a)C(b)=C(a)C(c)=0$ and 
    $ b=C(b) (a \vee c)$. By Theorem \ref{th203}, $b=C(b)(a+c)=C(b) a + C(b) c=C(b) c$. Thus $b \leq c$.    
    \end{proof}
    \begin{theorem} \label{th204} A p.q.-Baer $*$-ring with the order $\leq$ and the orthogonality $\perp$ is a
    quasi-orthomodular poset.     
    \end{theorem}
   \begin{proof} The proof follows from Theorem \ref{th203} and Lemmas \ref{lm204} and \ref{lm205}.    
   \end{proof}
   \begin{theorem} In a p.q.-Baer $*$-ring $R$, the initial segments $[0,m]=\{a\in  R~|~a \leq m  \}$
    are orthomodular posets. Furthermore, if $\perp_m$ is the local orthogonality of the initial segment $[0,m]$,
    then for all $a, b\in [0, m]$, $a \perp_m b$ if and only if $a\perp b$. Moreover, if $a \perp_m b$ and
      $a, b \leq n$, then $a \perp_n b$.    
   \end{theorem}
    \begin{proof} The proof follows from Theorems \ref{th202} and \ref{th204}.     
    \end{proof}
      
 \noindent {\bf Acknowledgment:} The first author gratefully
 acknowledges the University Grant Commission, New Delhi, India
 for the award of Teachers Fellowship under the faculty development
 program, during the $XII^{th}$ plan period (2012-1017).

$$\diamondsuit\diamondsuit\diamondsuit$$

\end{document}